\documentclass[12pt, reqno]{amsart}
\usepackage{amsfonts}
\usepackage{bbm}
\usepackage{amscd,amsfonts}
\usepackage{amssymb, eucal, amsfonts, amsmath, xypic, latexsym, tikz}
\usepackage{pifont}
\usepackage{mathrsfs,color}
\usepackage{amsthm,indentfirst,bm,fancyhdr,dsfont}
\usepackage{graphicx}
\usepackage[all]{xy}
\usepackage[CJKbookmarks=true]{hyperref}

\usepackage{mathrsfs}
\usepackage{amsmath}
\usepackage{amssymb}
\usepackage{hyperref}

\newtheorem{theorem}{Theorem}[section]
\newtheorem{lemma}[theorem]{Lemma}

\newtheorem{prop}[theorem]{Proposition}
\newtheorem{proposition}[theorem]{Proposition}
\newtheorem{corollary}[theorem]{Corollary}
\theoremstyle{definition}

\newtheorem{convention}[theorem]{Convention}
\newtheorem{example}[theorem]{Example}

\newtheorem{remark}[theorem]{Remark}

\numberwithin{equation}{section}

\def\ggg{\mathfrak{g}}

\def\gl{\mathfrak{gl}}

\def\cb{\mathcal{B}}
\def\Cg{{C\hskip-2pt\ggg}}
\def\CS{{C\hskip-2pt S}}

\def\ggg{\mathfrak{g}}

\def\ttt{\mathfrak{t}}
\def\hhh{\mathfrak{h}}

\def\nnn{\mathfrak{n}}
\def\uuu{\mathfrak{u}}

\def\bbf{\mathbb{F}}

\def\bba{{\mathbb{A}}}

\def\ug{\underline{g}}
\def\uv{\underline{v}}
\def\uGz{\underline{G_0}}

\def\sesi{\text{ss}}

\def\sfd{\textsf{d}}

\def\Lie{\mathsf{Lie}}

\def\Ad{\text{Ad}}

\def\GL{\text{GL}}
\def\Hom{\text{Hom}}
\def\Aut{\text{Aut}}
\def\Der{\text{Der}}

\def\Nor{\text{Nor}}

\def\id{\mathsf{id}}

{\vskip-\lastskip\medskip
  \noindent
  }%
{\qed\par\medskip
  }
\newcommand{\ff}{\mathbb{F}} 
\newcommand{\on}[1]{\operatorname{#1}}

\newcommand{\cw}{\mathscr{W}}

\newcommand{\co}{\mathcal{O}}
\newcommand{\cd}{\mathcal{D}} 

\def\GL{\text{\rm GL}}
\def\diag{\text{\rm Diag}}
\def\Nor{\text{\rm Nor}}
\def\Cent{\text{\rm Cent}}

\def\Ad{\mbox{Ad}}

\newcommand{\lieg}{\mathfrak{g}}

\newcommand{\ra}{\rightarrow} 
\newcommand{\inj}{\hookrightarrow} 
\newcommand{\mt}{\mapsto} 

\newcommand{\lieh}{\mathfrak{h}} 
\newcommand{\liel}{\mathfrak{l}} 
\newcommand{\lieb}{\mathfrak{b}} 
\newcommand{\liep}{\mathfrak{p}} 
\newcommand{\liet}{\mathfrak{t}} 

\newcommand{\iso}{\simeq} 

\newcommand{\ds}{\oplus} 

\newcommand{\p}{\partial} 

\newcommand{\she}[1]{\mathcal{#1}} 
\newcommand{\nc}{\she{N}} 
\newcommand{\nilg}[1]{\she{N}(\lieg_{#1})}

\newcommand{\bb}[1]{\mathbb{#1}} 

\newcommand{\pr}[1]{#1^{\prime}}

\begin{document}



\thanks{This work is partially supported by the National Natural Science Foundation of China (Grant Nos. 12071136, 11671138 and 11771279), Shanghai Key Laboratory of PMMP (No. 13dz2260400), the Fundamental Research Funds of Yunnan Province (Grant No. 2020J0375) and the Fundamental Research Funds of YNUFE (Grant No. 80059900196).}

\title[Semi-reductive groups]{On Chevalley restriction theorem for semi-reductive algebraic groups and its applications}

\author{Ke Ou, Bin Shu and Yu-Feng Yao}

\address{Department of Statistics and Mathematics, Yunnan University of Finance and Economics, Kunming, 650221, China.}\email{keou@ynufe.edu.cn}
\address{School of Mathematical Sciences, East China Normal University, Shanghai, 200241, China.} \email{bshu@math.ecnu.edu.cn}
\address{Department of Mathematics, Shanghai Maritime University, Shanghai, 201306, China.}\email{yfyao@shmtu.edu.cn}

\begin{abstract} An  algebraic group is called  semi-reductive if it is a semi-direct product of a reductive subgroup and the unipotent radical.  Such a semi-reductive algebraic group  naturally arises and also plays a key role in the study of modular representations of non-classical finite-dimensional simple Lie algebras in positive characteristic, and some other cases. Let $G$ ba a connected semi-reductive algebraic group over an algebraically closed field $\ff$ and $\ggg=\Lie(G)$. It turns out that $G$ has many same properties as reductive groups, such as the Bruhat decomposition.
In this note,  we obtain an analogue of classical Chevalley restriction theorem for $\ggg$, which says that the $G$-invariant ring  $\ff[\ggg]^G$ is a polynomial ring if $\ggg$ satisfies a certain ``posivity" condition suited for lots of cases we are interested in.
As applications, we further investigate the nilpotent cones and  resolutions of singularities for  semi-reductive Lie algebras.
\end{abstract}
\maketitle

\section{Introduction}
Let $ \ff $ be an algebraically closed field. An algebraic group $G$ over $\bbf$ is called semi-reductive if $G=G_0\ltimes U$ with $G_0$ being a reductive subgroup, and $U$ the unipotent radical. Let $\ggg=\Lie(G)$, $\ggg_0=\Lie(G_0)$ and $\uuu=\Lie(U)$, then $\ggg=\ggg_0\oplus\uuu$, which is called a semi-reductive Lie algebra.

Preliminarily, the motivation of initiating the study of semi-reductive algebraic groups and Lie algebras comes from the study of irreducible representations of finite-dimensional non-classical simple Lie algebras in prime characteristic (see Example \ref{ex1}).  In the direction of this new study, there have been some interesting mathematical phenomenons revealing (to see \cite{Ren}, \cite{RenShu}, \cite{SYX}, \cite{Xue}).

 In the present paper, we focus on establishing an analogue of the classical Chevalley restriction theorem in the case of semi-reductive algebraic groups.
 
Recall that for a connected reductive algebraic group $L$ over $\bbf$ with $ \liel=\Lie(L)$, and a maximal torus $T$ of $L$ with $\lieh=\Lie(T)$. We have the Weyl group $ \mathscr{W}=\on{Nor}_{L}(T)/\on{Cent}_{L}(T)$.
The classical Chevalley restriction theorem asserts that the restriction
homomorphism $ \ff[\liel]\ra \ff[\lieh] $ induces an isomorphism of algebras of invariants: $  \ff[\liel]^{{L}}\cong\ff[\lieh]^{\mathscr{W}}$ where  $\ff[ \liel ] $ (resp. $ \ff[\lieh] $) denotes the algebra of polynomial functions on $ \liel $ (resp. $ \lieh $) (see \cite[Proposition 7.12]{Jan} if $ \on{ch}(\ff)>0 $ and \cite[\S 23.1 and Appendix 23]{Hum2} for semisimple $ \liel $ if $ \on{ch}(\ff)=0 $). The Chevalley restriction theorem plays a key role in the study of representations of reductive Lie algebras. It has been generalized in various ways and cases (see \cite{LR}, \cite{Pre} etc.). However, it may fail if $ \liel $ is replaced with an arbitrary algebraic Lie algebra $ \mathfrak{q}. $ In particular, one has to distinguish the adjoint and coadjoint representations of $ \mathfrak{q}. $ See \S\ref{4} for a concrete example.

In this note, we will generalize the Chevalley restriction theorem to the case of a semi-reductive algebraic group $ G $ if $G$ satisfies a certain ``posivity" condition (see Convention \ref{positive} for the precise definition).

The generalization of the Chevalley restriction theorem seems to be quite useful in understanding the structure of semi-reductive Lie algebras. As applications, we investigate the nilpotent cones and resolutions of singularities for semi-reductive Lie algebras. In particular, Propositions \ref{nil cone 3}, \ref{5.3} and \ref{5.9} provide the structure of nilpotent cone $ \nc $, Steinberg maps and resolution of singularities of $ \nc $ respectively.

Our paper is organized as follows. In \S2, we recall some known results, and investigate the structures of semi-reductive algebraic groups as well as semi-reductive Lie algebras. In \S3, the Chevalley restriction theorem for semi-reductive Lie algebras $ \lieg $ are presented provided $ \lieg $ satisfied a certain ``posivity" condition and some restrictions on the characteristic of $ \ff. $ As the first application, we illustrate the difference between the adjoint and coadjoint representations of $ \lieg $ in \S4. As the second application, we study the nilpotent cone of $ \lieg $ and its Springer resolution  in the final section.

\section{Semi-redutive groups and semi-reductive Lie algebeas}
\subsection{Notions and assumptions}
Throughout this paper, all vector spaces and varieties are defined over an algebraically closed field $\bbf$.

$G=G_0\ltimes U$ be a semi-reductive group with $G_0$ being a reductive subgroup, and $U$ the unipotent radical. Let $\ggg=\Lie(G)$, $\ggg_0=\Lie(G_0)$ and $\uuu=\Lie(U)$, then $\ggg=\ggg_0\oplus\uuu$. Let $ \on{pr}: \lieg  \ra  \lieg_0,\ \pi: \lieg  \ra  \uuu $ be the projections. In the following, we give some examples of  semi-reductive Lie algebras (resp. semi-reductive algebraic groups).

\begin{example}\label{ex1} (The non-negative graded part of a restricted Lie algebra of Cartan type over $\bbf$ of prime characteristic $p>0$)
Let $\mathscr{A}(n)=\bbf[T_1,\ldots,T_n]\slash (T_1^p,\ldots,
T_n^p)$, a quotient of the polynomial ring by the ideal generated by $T_i^p$, $1\leq i\leq n$. Set $\mathscr{L}=\text{Der}(\mathscr{A}(n))$. Then $\mathscr{L}$ is a simple Lie algebra unless $n=1$ and $p=2$ (this simple Lie algebra is usually called a generalized Witt algebra, denoted by $W(n)$). Associated with the degrees of polynomial quotients, $\mathscr{L}$ becomes a graded algebra $\mathscr{L}=\sum_{d\geq -1}\mathscr{L}_{[d]}$. Then $\mathscr{L}$ has a filtration $$\mathscr{L}=\mathscr{L}_{-1}\supset\mathscr{L}_0\supset\cdots\supset\mathscr{L}_{n(p-1)-1}\supset 0.$$
where $$\mathscr{L}_{i}=\bigoplus_{j\geq i}\mathscr{L}_{[j]},\,\,-1\leq i\leq n(p-1)-1.$$
Let $ G=\on{Aut}(\mathscr{L}) $ be the automorphism group of $ \mathscr{L} $ and $\ggg=\Lie(G)$. 
 Then $G=\GL(n,\bbf)\ltimes U$ and $\ggg=\mathscr{L}_0$ with $\Lie(\GL(n,\bbf))=\mathscr{L}_{[0]}(\cong\mathfrak{gl}(n,\bbf))$, $\Lie(U)=\mathscr{L}_1=\sum_{d\geq 1}\mathscr{L}_{[d]}$ (see \cite{LN,Wil}). So $G$ is a semi-reductive group and $\mathscr{L}_0$ is a semi-reductive Lie algebra. According to the known results (cf. \cite{Nakano}, \cite{Shen}, \cite{Sk}, \cite{ShuYao1}, \cite{ShuYao2} and \cite{Zhang}, {etc}),  the representation theory of $W(n)$ along with other Cartan type Lie algebras are, to most extent, determined by that of its distinguished maximal filtered subalgebra $\mathscr{L}_0$. So the study of semi-reductive Lie algebra  $\mathscr{L}_0$ becomes a crucial important topic.

More generally, apart from $W(n)$, there are another three series of Cartan type restricted simple Lie algebras $S(n), H(n), K(n)$ (see \cite{SF}, \cite{St}, or the forthcoming Example \ref{lem pst ex3}).  Each of them is endowed with the corresponding graded structure.
Similarly, one can consider $\ggg=\sum_{d\geq 0}X(n)_{[d]}$ for $X\in \{S,H,K\}$ and $G=\Aut(X(n))$.
The corresponding semi-reductive groups and Lie algebras also appear as above.
\end{example}

\begin{example}\label{ex2} (Parabolic subalgebras of reductive Lie algebras in any characteristic) Let $\tilde G$ be a reductive algebraic group over $\bbf$ and $G$ be a  parabolic subgroup of $\tilde G$. Then $G$ is an semi-reductive group.
\end{example}

\begin{example} \label{ex22} Let $G$ be a connected reductive algebraic group over $\bbf$ satisfying that the characteristic of $\bbf$ is good for $G$, in the sense of \cite[\S3.9]{Hum2}.  For any given nilpotent $X\in\ggg$, let $G_X$ be the centralizer of $X$ in $G$, and  $\ggg_X=\Lie(G_X)$. By \cite[\S5.10-\S5.11]{Jan},  $G_X=C_X\ltimes R_X$ is semi-reductive.
\end{example}

\begin{example}\label{ex3} (Enhanced reductive algebraic groups) Let $G_0$ be a connected reductive algebraic group over $\bbf$, and $(M,\rho)$ a finite-dimensional representation of $G_0$ with representation space $M$ over $\bbf$.  Here and further, representations of an algebraical group always mean its rational representations.   Consider the product variety $G_0\times M$. Regard $M$ as an additive algebraic group. The variety $G_0\times M$ is endowed with a cross product structure denoted by $G_0\times_\rho M$,  by defining for any $(g_1,v_1), (g_2,v_2)\in  G_0\times M$
\begin{align}\label{product}
(g_1,v_1)\cdot (g_2,v_2):=(g_1g_2, \rho(g_1)v_2+v_1).
\end{align}
Then by a straightforward computation it is easily known that  $\underline{G_0}:=G_0\times_\rho M$ becomes a group with unity $(e,0)$ for the unity $e\in G_0$, and $(g,v)^{-1}=(g^{-1}, -\rho(g)^{-1}v)$.
Then $G_0\times_\rho M$ has a subgroup $G_0$ identified with $(G_0, 0)$ and a subgroup $M$ identified with  $(e, M)$.
Furthermore,  $\underline{G_0}$ is connected since $G_0$ and $M$ are irreducible varieties. We call $\underline{G_0}$ \textsl{ an enhanced reductive algebraic group associated with the representation space $M$}. What is more, $G_0$ and $M$ are closed subgroups of $\underline{G_0}$, and $M$ is a normal closed subgroup. Actually, we have $\ug\uv\ug^{-1}=\rho(g)v$. Here and further $\ug$ stands for $(g,0)$ and $\uv$ stands for $(e,v)$.
Set $\ggg_0=\Lie(G_0)$. Then  $(\mathsf{d}(\rho),M)$ becomes a representation of $\ggg_0$.  Naturally, $\Lie(\underline{G_0})=\ggg_0\oplus M$, with Lie bracket
 $$[(X_1,v_1),(X_2,v_2)]:=([X_1,X_2], \mathsf{d}(\rho)(X_1)v_2-\mathsf{d}(\rho)(X_2)v_1),$$
 which is called an enhanced reductive Lie algebra.

Clearly, $\uGz$ is a semi-reductive group with $M$ being the unipotent radical.
\end{example}


\subsection{}\label{1.2} In the sequel, we always assume that $G=G_0\ltimes U$ is a connected semi-reductive algebraic group  over  an algebraically closed field $\bbf$ where $G_0$ is a connected reductive subgroup and $U$ the unipotent radical. Let $\ggg=\Lie(G)$ be the  Lie algebra of $G$.
Recall that a Borel subgroup (resp. Borel subalgebra) of $ G $ (resp. $ \lieg $) is a maximal solvable subgroup (resp. subalgebra) of $ G $ (resp. $ \lieg $).
In the following we will illustrate the structure of Borel subgroups of $G$.
\begin{lemma}\label{startpoint}  
The following statements hold.
\begin{itemize}
\item[(1)] Suppose $B$ is a Borel subgroup of $G$. Then $B\supset U$. Furthermore, $B_0:=B\cap G_0$ is a Borel subgroup of $G_0$, and $B=B_0\ltimes U$.
\item[(2)] Any maximal torus $T$ of $G$ is conjugate to a maximal torus $T_0$ of $G_0$.
\end{itemize}
\end{lemma}

\begin{proof} (1) As $U$ is the unipotent radical of $G$,  $BU$ is still a closed subgroup containing $B$. We further assert that $BU$ is solvable. Firstly,  by a straightforward computation we have that the $i$th derived subgroup $\cd^i(BU)$ is contained in $\cd^{i}(B)U$. By the solvableness of $B$, there exists some positive integer $t$ such that $\cd^{t}(B)$ is the identity group $\{e\}$. So $\cd^t(BU)\subset U$. Secondly,  as $U$ is unipotent, and then solvable. So there exists some positive integer $r$ such that $\cd^r(U)=\{e\}$. Hence $\cd^{t+r}(BU)=\{e\}$.  The assertion is proved.

The maximality of the solvable closed subgroup $B$ implies that  $BU=B$. This is to say, $U$ is contained in the unipotent radical $B_u$ of $B$. Set $B_0=B\cap G_0$ which is clearly a closed solvable subgroup of $G_0$. By the same token as above, $B_0U$ is a closed solvable subgroup of $B$. On the other hand, for any $b\in B$, by the definition of semi-reductive  groups  we have $b=b_0u$ for some $b_0\in G_0$ and $u\in U$. As $U\subset B$, we have further that $b_0=bu^{-1}\in B_0 $. It follows that $B$ is contained in $B_0\ltimes U$. The remaining thing is to prove that $B_0$ is a Borel subgroup of $G_0$. It is clear that $B_0$ is a solvable closed subgroup of $G_0$. If $B_0$ is not a maximal solvable closed subgroup of $G_0$. We may suppose $B_0'$ is a larger one properly containing $B_0$. Then the solvable closed subgroup $B_0'U$ of $G$ contains $B$ properly. It contradicts with the maximality of $B$. Hence $B_0$ is really maximal. Summing up, the statement in (1) is proved.

(2) Note that the maximal tori in $G$ are all conjugate (see \cite[11.3]{Bor}). This statement follows from (1).
\end{proof}

{ As a corollary to Lemma \ref{startpoint}, we have the following fact for semi-reductive Lie algebras. }

\begin{lemma}\label{max torus} For a semi-reductive  group $G$ and $\ggg=\Lie(G)$,  all maximal tori of $\ggg$ are conjugate under adjoint action of $G$. Moreover, $\ggg_\sesi=\bigcup_{T}\Lie(T)$ where $\ggg_{\sesi}$ is the set consisting of all semisimple elements and $ T $ runs over all maximal tori of $ G $.
\end{lemma}

By Lemmas \ref{startpoint} and \ref{max torus}, we can choose a maximal torus $T$ of $G$, which lies in $G_0$ without loss of generality. By \cite[\S8.17]{Bor},  we have the following decomposition of root spaces associated with  $\ttt:=\on{Lie}(T)$
\begin{align}\label{root decomp}
\ggg=\ttt\oplus\sum_{\alpha\in \Phi(G_0,T)}(\ggg_0)_\alpha\oplus\sum_{\alpha\in \Phi(U,T)}\uuu_\alpha
\end{align}
 where $\Phi(U,T)$ is the subset of $X(G,T)$ satisfying that for $\uuu:=\Lie(U)$,
$$\uuu=\sum_{\alpha\in \Phi(U,T)}\uuu_\alpha.$$

In the consequent arguments, the root system $\Phi(G_0,T)$ will be denoted by $\Delta$ for simplicity, which is actually independent of the choice of $T$. We fix a positive root system $\Delta^+:=\Phi(G_0,T)^+$. The corresponding Borel subgroup will be denoted by  $B^+$ which contains $T$, and the corresponding simple system is denoted by $\Pi$.

\subsection{}\label{1.3}

The following facts are clear.

\begin{lemma}\label{fund-1} Let $G$ be a connected semi-reductive  group, and $T$  a  maximal torus  of $G_0$.
\begin{itemize}

\item[(1)] Set $\cw(G,T):=\Nor_G(T)\slash \Cent_G(T)$. Then $\cw(G,T)\cong \cw$, where $\cw$ is the Weyl group of $G_0$. This is to say, $G$ admits the Weyl group $\cw$ coinciding with the one of $G_0$.

\item[(2)] Set $\{\dot w\mid w\in W\}$ be a set of representatives in $N_G(T)$ of the elements of $\cw$. Denote by $C(w)$ the double coset $B^+\dot w B^+$, and  by $C_0(w)$ the corresponding component $B^+_0\dot w B^+_0$ of the Bruhat decomposition of $G_0$. Then  for any $w\in \cw$
\begin{align}\label{cells}
C(w)=C_0(w)\ltimes U.
\end{align}
\item[(3)] Let $w=s_1\cdots s_h$ be a reduced expression of $w\in \cw$, with $s_i=s_{\gamma_i}$ for $\gamma_i\in \Pi$. Then $C(w)\cong \bba^h\times B^+$.
\item[(4)] For $w\in \cw$, set $\Phi(w):=\{\alpha\in \Delta^+\mid w\cdot \alpha\in -\Delta^+\}$, and $U_w:=\Pi_{\alpha\in \Phi(w)}U_\alpha$, then $U_{w^{-1}}\times B^+\cong C(w)$ via sending $(u,b)$ onto $u\dot w b$.
\item[(5)] (Bruhat Decomposition) We have $G={\dot\bigcup}_{w\in \cw} B^+\dot w B^+$.
Therefore, for any $g\in G$, $g$ can be written uniquely in the form $u\dot w b$ with $w\in \cw$, $u\in U_{w^{-1}}$ and $b\in B^+$.
\end{itemize}
\end{lemma}

\begin{proof}
(1) Note that $U$ is the unipotent radical of $G$, and $G=G_0\ltimes U$. We have $\Nor_G(T)=\Nor_{G_0}(T)$ and $\Cent_G(T)=\Cent_{G_0}(T)$. The statement is proved.

 (2) Note  that  $U\dot w=\dot w U$ and $B^+=B^+_0U=UB^+_0$ for $B^+_0=B^+\cap G_0$. We have
\begin{align}
C(w)&=B^+_0U\dot w B^+_0U\cr
&=B^+_0\dot w UB^+_0U\cr
&=B^+_0\dot wB^+_0 U\cr
&=C_0( w)\ltimes U.
\end{align}

(3) and (4) follow from (2) along with \cite[Lemma 8.3.6]{Sp}.

(5) follows from (2) along with the Bruhat decomposition Theorem for $G_0$ (\cite[Theorem 8.3.8]{Sp}).
\end{proof}

\begin{proposition}\label{basic for fiber bundle} Let $G$ be a connected semi-reductive  group. Keep the notations as in Lemma \ref{fund-1}. The following statements hold.
\begin{itemize}
\item[(1)] The group $G$ admits a unique open double coset $C(w_0)$, where $w_0$ is the longest element of $\cw$.
\item[(2)] For any given Borel subgroup $B^+$, let $\cb$ denote the homogeneous space $G\slash B^+$, and $\pi:G\rightarrow \cb$ be the canonical morphism. Then $\pi$ has local sections.

\item[(3)] For a rational  $B^+$-module $M$, there exists a fibre bundle over $\cb$ associated with $M$,  denoted by $G\times^{B^+} M$.

\end{itemize}
\end{proposition}

\begin{proof} (1) Note that $C(w_0)$ is a $B^+\times B^+$-orbit in $G$ under double action, and $C(w_0)$ is open in its closure. Thanks to Lemma \ref{fund-1}(4), we have $$\dim C(w_0)=\#\{\Delta^+\}+\dim B^+=\dim G.$$
On the hand, $G$ is irreducible, so $G=\overline {C(w_0)}$. It follows that $C(w_0)$ is an open subset of $G$. By the uniqueness of the longest element in $\cw$ and Lemma \ref{cells}(3)(5), we have that $C(w_0)$ is a unique open double coset.

(2) According to \cite[\S5.5.7]{Sp}, we only need to certify that (a) $\cb$ is covered by open sets $\{U\}$; (b) each of such open sets has a section, i.e. a morphism $\sigma: U\rightarrow \cb$ satisfying $\pi\circ\sigma=\id_U$. Thanks to \cite[Theorem 5.5.5]{Sp},  $\pi$ is open and separable. Let $X(w)=\pi(C(w))$. By (1), $X(w_0)$ is an open set of $\cb$. By Lemma \ref{fund-1}(4),  $\pi$ has a section on $X(w_0)$. Note that $\{gX(w_0)\mid g\in G\}$ constitute of open covering of $\cb$. Using the translation, the statement follows from the argument on $X(w_0)$.

(3) follows from (2) and \cite[Lemma 5.5.8]{Sp}.
\end{proof}

\subsection{Regular cocharacters and the positivity condition}\label{reg}
Keep the notations and assumptions as above subsections. Let $T$ be a given maximal torus of $G_0$. By \cite[\S12.2]{Bor}, there is a regular semisimple element $t\in T$ such that $\alpha(t)\ne 1$ for any $\alpha\in \Phi(G_0,T)$. Furthermore, by \cite[Lemma 3.2.11]{Sp}, $X_*(T)\otimes \bbf^\times\rightarrow T$ with $\tau\otimes a\in X_*(T)\otimes \bbf^\times$ mapping to $\tau(a)\in T$ gives rise to an isomorphism of abelian groups. So there is a regular cocharacter $\tau\in X_*( T)$ such that $t=\tau(a)$ for some $a\in \bbf^\times$. Thus, we have
\begin{align}\label{regular coch1}
\langle\alpha,\tau\rangle\ne 0 \mbox{ for all }\alpha\in \Phi(G_0,T).
\end{align}
By \cite[\S7.4.5-\S7.4.6]{Sp}, we can further choose a regular cocharacter $\tau\in X_*(T)$ for the reductive group $G_0$ such that
\begin{align}\label{reg pos}
 \langle\alpha,\tau\rangle>0 \mbox{ for all }\alpha\in \Phi(G_0,T)^+.
 \end{align}

We turn to semi-reductive  case. Let $G$ be a semi-reductive  algebraic group and $\ggg=\Lie(G)$.
 \begin{convention}\label{positive} The group $G$ (or the Lie algebra $\ggg$) is said to satisfy  {\sl {the positivity condition}} if there exists a cocharacter $\chi\in X_*(T)$ such that $ \langle\alpha, \chi\rangle >0 \mbox{ for\,any }\alpha\in \Phi(G_0,T)^+\cup \Phi(U,T).$
\end{convention}

The positivity condition can be seen to valid in the following examples we are interested {\color{purple} in}.



\begin{example}\label{lem pst ex}
 Suppose $\ggg$ is  the Lie algebra of a semi-reductive algebraic group $G=\GL(V) \times_\nu V $ over $\bbf$ where $\nu$ is the natural representation of $\GL(V)$ (see Example \ref{ex3}). Then $\ggg$ satisfies the positivity condition.

Actually, suppose $ \on{dim}V=n. $ Fix a basis $\{v_i\mid 1\leq i\leq n\}$ of $V$ associated with the isomorphism $\GL(V)\cong\GL(n,\bbf).$ We
take the standard maximal torus $T$ of $\GL(n,\bbf)$  consisting of diagonal   invertible matrices. Set $B^+$ the Borel subgroup consisting of upper triangular invertible matrices, which corresponds to the positive root system. And take $\varepsilon_i\in X^*(T)$ with $\varepsilon_i(\diag(t_1,\cdots,t_n))=t_i$, $i=1,\cdots,n$. Then the character group $X^*(T)$ is just the free abelian group generated by $\varepsilon_i$, $i=1,\cdots,n$;
and the positive root system $\Phi(G_0,T)^+$ consists of all $\varepsilon_i-\varepsilon_j$ with  $1\leq i<j\leq n$. Set $\alpha_i=\varepsilon_i-\varepsilon_{i+1}$ for $i=1,\cdots,n-1$. Then $\Pi=\{\alpha_1,\cdots,\alpha_{n-1}\}$ becomes a fundamental root system of $\GL(V)$, and any positive root $\varepsilon_i-\varepsilon_j$ is expressed as $\sum_{s=i}^{j-1}\alpha_s$ for $i<j$.  Then by (\ref{product}), we have that $\Ad(\textbf{t})v_i=t_iv_i=\varepsilon_i(\textbf{t})v_i$ for $\textbf{t}=\diag(t_1,\cdots,t_n)\in T$. It follows that $\Phi(V,T)=\{\varepsilon_1,\varepsilon_2, \cdots, \varepsilon_n\}$ is the set of weights of $V$ with respect to $T$.

Set $\varpi:=\sum_{i=1}^n (n-i+1)\varepsilon_i$. The corresponding cocharacter is denoted by $\chi$, which is a regular cocharacter of $\GL(V)$. Then  $\langle \varepsilon_i,\chi\rangle=(n-i+1)$ and $\langle \alpha_j, \chi\rangle=1$ for $1\leq i\leq n$ and $1\leq j\leq n-1$. Hence $\GL(V)\times_\nu V$ satisfies the positivity condition \ref{positive}.
\end{example}

\begin{example}\label{lem pst ex2} 
 Suppose $\ggg$ is  the Lie algebra of a parabolic subgroup of a reductive algebraic group over $\bbf$ (see Example \ref{ex2}). Then the positivity condition is assured for  $\ggg$.

We can show this below. Suppose $P$ is a parabolic subgroup of a connected reductive algebraic group $G$. In this case, we can take a regular co-character $\chi$ of $G$ such that  $\langle\alpha,\chi\rangle >0$ for all positive root $ \alpha $ of $G$.
\end{example}

 \begin{example}\label{lem pst ex3} (Continued to Example \ref{ex1}) In this example, we  will further make some preliminary  introduction to  Cartan type Lie algebras in the classification of finite-dimensional simple Lie algebras over an algebraically closed filed of positive characteristic, which do not arise from simple algebraic groups (see \cite{KS, PreSt, St}). Suppose $\bbf$ is an algebraically closed field of characteristic $p>2$ for the time being. There are four types  of those Lie algebras,
 {\color{purple}and }each type of them contains  infinite series of ones. We will introduce the first two types: Witt type $W(n)$ ($\geq 1$) and Special type $S(n)$ ($n\geq 2$) (the remaining two types are Hamiltonian type $H(n)$ for even $n\geq 2$ and contact type $K(n)$ for odd $n\geq 3$,  see \cite{SF, St} for details).

    By definition, the restricted Lie algebras of Cartan type are $p$-subalgebras of the algebra of derivations of the truncated polynomial rings{
    $\mathscr{A}(n)=\bbf[T_1,\cdots,
    T_n]\slash (T_1^p,\ldots,T_n^p)$}, whose canonical generators are denoted by $x_i$, $i=1,\ldots n$.
Recall $W(n)=\Der(\mathscr{A}(n))$ (the $n$-th Jacobson-Witt algebra). Its restricted mapping is the standard $p$-th power of linear operators. Denote by $\partial_i$ the partial derivative with respect to the variable $x_i$ . Then $\partial_i$, $i=1,\ldots,n$ is a basis of the $\mathscr{A}(n)$-module $W(n)$. The Lie algebra $W(n)$ has  a restricted grading $W(n)=\bigoplus_{i=0}^{q} W(n)_{[i]}$ where $W(n)_{[i]}=\sum_{k=1}^n\mathscr{A}_{i+1}\partial_i$ associated with the natural grading of $\mathscr{A}$ arising from the degree of truncated polynomials, and $q=n(p-1)-1$. Let $\Omega$ denote the exterior differential algebra over $\mathscr{A}(n)$ in the usual sense. Then $\Omega(n)=\sum_{i=1}^n\Omega(n)^i$ with $\Omega(n)^i=\bigwedge^i\Omega(n)^1$ for $\Omega(n)^1:=\Hom_{\mathscr{A}(n)}(W(n),\mathscr{A}(n))$ which can be presented as $\Omega(n)^1=\sum_{i=1}^r\mathscr{A}(n)\sfd x_i$  where $\sfd $ is an $\bbf$-linear map from $\Omega^0(n):=\mathscr{A}(n)$ to $\Omega^1(n)$ via $\sfd f:E\mapsto E(f)$ for any $f\in \mathscr{A}(n)$ and $E\in W(n)$. The exterior differential algebra $\Omega(n)$ is canonically endowed with a $W(n)$-action (see \cite[\S4.2]{St}).

Associated with the volume differential form $\omega_S:=\sfd x_1\wedge\cdots\wedge \sfd x_n$, there is a (Cartan-type) simple Lie algebra of Special type $S(n):=S'(n)^{(1)}$ for $S'(n)=\{E\in W(n)\mid E\omega_S=0\}$.  Both $W(n)$ and $S(n)$ are simple Lie algebras (see \cite{KS, PreSt, SF, St}), the latter of which is a graded Lie subalgebra of the former. For $\ggg=W(n)$ or $S(n)$, the graded structure  $\ggg=\sum_{i\geq -1}\ggg_{[i]}$ gives rise to the filtration $\{\ggg_i\}$ with $\ggg_i=\sum_{j\geq i}\ggg_{[j]}$. Denote by $G=\Aut(\ggg)$. Then $G$ is compatible with the restricted mapping, and $G$ is connected provided that  $p\geq 5$ is additionally required for $\ggg=W(1)$. Furthermore,  the following properties are satisfied.
\begin{itemize}
\item[(1)] $G=\GL(n)\ltimes U$ where $U$  is the unipotent radical of G consisting of elements $\sigma\in G$: $(\sigma-\id_\ggg)(\ggg_i)\subset \ggg_{i+1}$. So $G$ is semi-reductive.

\item[(2)] $\Lie(G)=(\Cg)_0$ where
\begin{align}
\Cg:=\begin{cases} W(n)  &\text{ if }\ggg=W(n);\cr
\{E\in W(n)\mid E\omega_S\in \bbf\omega_S\}; &\text{ if } \ggg=S(n)
\end{cases}
\end{align}
which is a graded subalgebra of $W(n)$ (see \cite[Proposition 3.2]{LN}). Actually, $\CS(n)=S'(n)+\bbf x_1\partial_1$ (see \cite[\S4.2]{St}). So $\Cg=(\Cg)_{[0]}\oplus (\Cg)_1$ with $(\Cg)_{[0]}=W_{[0]}\cong\gl(n)$ and $(\Cg)_1\subset W(n)_1$.
\end{itemize}
 In particular, for $\ggg=W(n)$ we have $G=\GL(n)\ltimes U $ with $\Lie(G)=W(n)_0$, and $\on{Lie}(U)=W(n)_{1}$. Take a maximal torus $T$ and a positive Borel subgroup $ B^+ $ as Example \ref{lem pst ex}. Then $\Phi(G_0,T)^+=\{  \varepsilon_i-\varepsilon_j\mid 1\leq i<j\leq n\}$. For each $\textbf{t}=\diag(t_1,\cdots,t_n)\in T,$ $\Ad(\textbf{t})(x_1^{a_1}\cdots x_n^{a_n}\p_j)= (\sum_{i=1}^{n}a_i\varepsilon_i-\varepsilon_j) (\textbf{t})x_1^{a_1}\cdots x_n^{a_n}\p_j.$ Therefore,
$$ \Phi(U,T)=\left\{ \sum_{i=1}^{n}a_i\epsilon_i-\epsilon_j\mid 0\leq a_i\leq p-1\text{ and } \sum_{i=1}^{n}a_i\geq 2 \right\}. $$

Set $\varpi:=\sum_{i=1}^n (2n-i)\varepsilon_i$. The corresponding cocharacter is denoted by $\chi$, which is a regular cocharacter of $\GL(V)$. Then $\langle\varepsilon_i-\varepsilon_j,\chi\rangle=j-i$ for $1\leq i<j\leq n$ and
$$\langle  \sum_{i=1}^{n}a_i\varepsilon_i-\varepsilon_j,\chi\rangle= \sum_{i=1}^{n}a_i(2n-i)-(2n-j)\geq \sum_{i=1}^{n}a_in-(2n-1)\geq 2n-(2n-1)>0$$
for all $ \sum_{i=1}^{n}a_i\varepsilon_i-\varepsilon_j\in \Phi(U,T)$. Hence associated with $W(n)$,  $\Lie(G)$  satisfies the positivity condition \ref{positive}.

Associated with $S(n)$, $\Lie(G)$ still contains the reductive Lie subalgebra part isomorphic to $\gl(n)$ along with  the unipotent part $\Lie(U)\subset W(n)_1$. So the same arguments yield the assertion for $\Lie(G)$ associated with $S(n)$.

\end{example}

\begin{remark} Associated with $\ggg=H(n)$ and $K(n)$, $G=\Aut(\ggg)$ still turn out to be a connected semi-reductive group. However, the above argument is not available to the Hamiltonian algebra $H(n)$ and the contact algebra $K(n)$.
Indeed, for $ H(n)$ with $ n=2m, $ take  a maximal torus $ \lieh=\langle h_i:=x_i\partial_{i}-x_{i+m}\partial_{i+m} \mid 1\leq i\leq m\rangle $, and $ \alpha_i $ the fundamental weight with $ \alpha_i(h_j)=\delta _{ij} $ for $1\leq i, j\leq m$. Note that both $ x_1^2\p_{m+1}$ and $x_{m+1}^2\p_1 $ lie in $ H(n)_{[1]},$ so that $ \pm3\alpha_1\in \Phi(U,T). $ Hence, the positivity condition fails for $ H(n). $ While for $K(n)$ with $ n=2m+1, $ take  a maximal torus $ \lieh=\langle h_i:=x_i\partial_{i}-x_{i+m}\partial_{i+m}, h_{m+1}:=\sum_{j=1}^{2m}x_j\partial_j+2x_{2m+1}\partial_{2m+1} \mid 1\leq i\leq m\rangle $, and $ \alpha_i $ the fundamental weight with $ \alpha_i(h_j)=\delta _{ij} $ for $1\leq i, j\leq m+1$. Note that both $3x_1^2x_{2m+1}^{\frac{p-1}{2}}\p_{m+1}+x_1^3x_{2m+1}^{\frac{p-3}{2}}(\sum_{j=1}^{2m}x_j\partial_j-x_{2m+1}\partial_{2m+1})$ and $3x_{m+1}^2x_{2m+1}^{\frac{p-1}{2}}\p_1+x_{m+1}^3x_{2m+1}^{\frac{p-3}{2}}(\sum_{j=1}^{2m}x_j\partial_j-x_{2m+1}\partial_{2m+1}) $ lie in $ K(n)_{[p]},$ so that $ \pm3\alpha_1\in \Phi(U,T). $ Hence, the positivity condition fails for $ K(n). $
\end{remark}

Analogue to \cite[Proposition 2.11]{Jan}, we have the following paralleling result for a semi-reductive algebraic group.

\begin{proposition}\label{semisimple orbit}
Suppose that $G=G_0\ltimes U$ is a semi-reductive  algebraic group over $\bbf$ with a maximal torus $T$ and satisfies the positivity condition \ref{positive}. Let $\ggg=\Lie(G)$ and $\ggg_0=\Lie(G_0)$.  Then for any $X\in \ggg$ with the Jordan-Chevalley decomposition $X=X_s+X_n$,  we have $X_s\in \overline{\co_X}$. In particular, if $ X $ is nilpotent, then $ 0\in \overline{\co_X}. $
\end{proposition}

\begin{proof} By Lemmas \ref{startpoint} and \ref{fund-1}, any $X\in\ggg$ under Ad$(G)$-conjugation lies in a fixed Borel subalgebra. We first fix a Borel subgroup $B=B_0\ltimes U$ such that the maximal torus $T$ is contained in $B_0\subset G_0$, and   $\Lie(B)=\liet+\sum_{\alpha\in \Phi(G_0,T)^+}(\ggg_0)_\alpha+\sum_{\alpha\in \Phi(U,T)}\uuu_\alpha$ (keeping in mind  (\ref{root decomp})). Without loss of generality, we might as well suppose  $X\in \Lie(B)$. Consider its Jordan-Chevally decomposition $X=X_s+X_n$. Then  $ X_s,X_n\in \Lie(B)$.
By assumption, we can write
\begin{align}\label{JC decomp}
 X=X_s+ \sum_{\alpha\in\Phi(G_0,T)^+} X_{{\alpha}} +\sum_{\beta\in \Phi(U,T)}X_\beta
 \end{align}
 where {$ X_s\in \ttt, $} $X_{{\alpha}}\in (\ggg_0)_\alpha$ for $\alpha\in \Phi(G_0,T)^+$ and $X_\beta\in \uuu_\beta$ for $\beta\in \Phi(U,T)$.

Then $\Ad(t)(X)=X_s+ \sum_{\gamma\in R^+ } {\gamma}(t) X_{\gamma}$ for all $t\in  T$,  with $R^+:=\Phi(G_0,T)^+\cup \Phi(U,T)$.  Suppose $ \chi $ is a cocharacter satisfies the positivity condition \ref{positive}. Then for all $ a\in \bbf^\times$, up to conjugation under $G$-action
	\[ \Ad(\chi(a))(X)=X_s  + \sum_{ \gamma\in R^+} a^{\langle \gamma,\chi\rangle} X_{\gamma} \]
with $\langle\gamma, \chi\rangle>0$ for all $\gamma\in R^+$, which shows that we can extend the map $ a\mt \Ad (\chi(a))X $ to a morphism $ \bbf\ra {\ggg} $ by $ 0\mt X_s$.  It follows that $ X_s\in \overline{G\cdot X}$ as claimed.
\end{proof}

\section{Chevalley Restriction Theorem}
From now on, we always assume that $G=G_0\ltimes U$ (resp. $ \lieg=\lieg_0\oplus \uuu$) is a semi-reductive  algebraic group (resp. semi-reductive  Lie algebra) over $ \ff $ of rank $ n $.


\subsection{}
Let $\cw$ be the Weyl group of $G$ relative to $T$, i.e., $\cw=N_G(T)/\text{C}_G(T)$, which can be identified with $N_{G_0}(T)/\text{C}_{G_0}(T)$. For each $ w\in \cw, $ denote by $ \dot{w} $ the presentation of $ w $ in $ N_G(T). $
Thanks to \cite[\S 7.12]{Jan} and \cite[\S 11.12]{Bor}, the following lemma holds.

\begin{lemma}\label{above lem}
	For a given maximal torus $T$, let $\lieh:=\Lie(T)$. Take $h\in \lieh$. Then	\[ \text{Ad}(G)(h)\cap \lieh =\{ \dot{w}(h)| w\in \cw \}. \]
\end{lemma}

\begin{lemma}(\cite[Proposition 7.12]{Jan})\label{reductive che res them}
	The map $f\longmapsto f|_{\hhh}$ induces an injective homomorphism $\Phi_0:\,\ff[\lieg_0]^{G_0}\longrightarrow \ff[\lieh]^\cw$ of $\bbf$-algebras.
	This homomorphism is an isomorphism if $ \text{char}(\bbf)\neq 2 $ or if $ \text{char}(\bbf) =2 $ and $ \alpha\notin 2X(T) $ for all roots $ \alpha $ of $G_0$ relative to $T$.
\end{lemma}

Consider the canonical projection $ \on{pr}: \lieg  \ra  \lieg_0$  and the imbedding $\on{i}:\lieg_0\ra \lieg$ for a given semi-reductive Lie algebra $\ggg=\ggg_0\oplus \uuu$. Then $\on{pr}$ and $\on{i}$ are  both  morphisms of algebraic varieties.

\begin{lemma}\label{pr} The comorphism
	$ \on{pr}^{*}:\bbf[\lieg_{0}] \ra \bbf[\lieg] $ induces an injective homomorphism $ \on{pr}^{*}:\bbf[\lieg_{0}]^{G_0} \ra \bbf[\lieg]^G, $ while
	$ \on{i}^{*}:\bbf[\lieg] \ra \bbf[\lieg_{0}] $ induces a surjective homomorphism $ \on{i}^{*}:\bbf[\lieg]^G\ra \bbf[\lieg_{0}]^{G_0}. $
\end{lemma}
\begin{proof}
	For each $ f\in  \bbf[\lieg_{0}]^{G_0}, $ denote $ F=\on{pr}^*(f)\in \ff[\lieg]. $ Equivalently, $ F (D+u) = f(D) $, where $ D\in {\lieg}_{0},\ u\in\uuu$.
	For any $ \sigma\in G$, we have $\sigma=\sigma_1\cdot \sigma_2 $ where $ \sigma_1\in G_0 $ and $ \sigma_2\in U. $ Then $ \sigma\cdot(D+u)=\sigma_1\cdot D+\pr{u} $, where $ \pr{u}\in \mathfrak{u}. $  Moreover,
	\[(\sigma^{-1}\cdot F)(D+u)=F(\sigma (D+u))=F(\sigma_1 D+\pr{u}) =f(\sigma_1 D)=f(D)=F(D+u). \]
	Namely, $ F $ is $ G $-invariant and $ \on{pr}^{*}:\bbf[\lieg_{0}]^{G_0} \ra \bbf[\lieg]^G $ is well-defined.
	
	Since $ \on{i}^* $ is the restriction map, it is easy to verify that $ \on{i}^{*}:\bbf[\lieg]^G\ra \bbf[\lieg_{0}]^{G_0} $ is well-defined.
	
	Note that $ \on{pr}\circ \on{i}$ is the identity map on $\lieg_0, $. Therefore, $ \on{i}^*\circ\on{pr}^*$ is the identity map on $ \ff[\lieg_0]^{G_0}. $ As a corollary, $ \on{pr}^* $ is injective and $ \on{i}^* $ is surjective.
\end{proof}
It follows from \cite{Che} that $ \ff[\lieh]^\cw $ is generated by $ n $ algebraically independent homogeneous elements (and the unit). Then $ \ff[\lieg_0] ^{G_0} $ is a polynomial algebra with $ n $ algebraically independent homogeneous polynomials
$\{f_i\mid 1\leq i\leq  n\} \subset \bbf[\lieg_{0}]$.
Denote $ F_i:= \on{pr}^{*}(f_i) $ for all $ i=1,\cdots,n. $ Then all $ F_i $ are homogeneous and algebraically independent, and $ \on{pr}^*(\ff[\lieg_0] ^{G_0}) = \ff[F_1,\cdots,F_n]. $

\begin{theorem}\label{che res thm}
	Let $G$ be a semi-reductive  algebraic group satisfying the positivity condition \ref{positive}.  Then the map $f\longmapsto f|_{\hhh}$ induces an injective homomorphism $\Phi:\,\ff[\lieg]^G\longrightarrow \ff[\lieh]^\cw$ of $\bbf$-algebras.
	This homomorphism is an isomorphism if $ \text{char}(\bbf)\neq 2 $ or if $ \text{char}(\bbf) =2 $ and $ \alpha\notin 2X(T) $ for all roots $ \alpha $ of $G_0$ relative to $T$.
\end{theorem}

\begin{proof}
	Note that $ \Phi=\Phi_0\circ  \on{i}^*. $ For each $ X\in\ggg, $ let $ X=X_s+X_n $ be the Jordan decomposition. It follows from Proposition \ref{semisimple orbit} that
	$f(X)=f(X_s)$ for all  $f\in \ff[\lieg]^G.$
	By Lemma \ref{max torus}, there is $ h\in \lieh $ such that $ X_s\in G\cdot h. $
	By Lemma \ref{above lem}, $G\cdot h\cap\hhh=\{\dot{w}(h)\mid w\in \cw\}$ for all $h\in\hhh$.
	It follows that each $f\in \ff[\lieg]^G$ is determined by $f|_{\hhh}$, which is invariant under the action of $\cw$.  This implies that $\Phi$ is injective.
	
	On the other hand, by Lemmas \ref{reductive che res them} and \ref{pr}, $ \Phi $ is surjective if $ \text{char}(\bbf)\neq 2 $ or if $ \text{char}(\bbf) =2 $ and $ \alpha\notin 2X(T) $ for all roots $ \alpha $ of $G_0$ relative to $T$.
\end{proof}

\begin{corollary}\label{2.5}
	Suppose $ G $ is a semi-reductive algebraic group of  rank $n$  satisfying the positivity condition \ref{positive}, and $ \text{char}(\bbf)\neq 2 $ or if $ \text{char}(\bbf) =2 $ and $ \alpha\notin 2X(T) $ for all roots $ \alpha $ of $G_0$ relative to $T$.  Then $ \ff[\lieg]^G $ is a polynomial algebra generated by $n$ algebraically independent homogeneous polynomials.
\end{corollary}

\subsection{Eamples} Theorem \ref{che res thm} entails the following good properties for those important examples we are interested and presented before.

\begin{corollary}\label{2.6}
	Let $\ggg$ be any one from the following list:
	\begin{itemize}
		\item[(1)] the Lie algebra of an semi-reductive algebraic group $G=\GL(V) \times_\nu V $ over $\bbf$ where $\nu$ is the natural representation of $\GL(V)$ (see Example \ref{lem pst ex}).
		\item[(3)] the Lie algebra of a parabolic subgroup $ P $ of a reductive algebraic group over $\bbf$ with decomposition $ P=L\ltimes U $ such that the unipotent radical $ U $ contains no trivial weight space as an $ L $ module (see Example \ref{lem pst ex2}).
\item[(3)] $\Lie(G)$ associated with $W(n)$  and $S(n)$ over $\bbf$ (see Example \ref{lem pst ex3}).	

\end{itemize}
Then the invariants of $\ff[\lieg]$ under $G$-action must be a polynomial ring. More precisely,
$ \ff[\lieg]^G\cong \ff[\lieh]^\cw$ as $\bbf$-algebras, where $\cw$ is the Weyl group of $\ggg$. In particular, $\cw$ coincides with the $n$-th symmetry group $\frak{S}_n$ in Case (1) and Case (3).
\end{corollary}

\section{Application I: Comparison with the center of enveloping algebra of the enhanced reductive Lie algebra  $\underline{\gl(V)}$}\label{4}
Suppose $ \on{ch}(\ff)=0 $ in this section. Let $G=\GL(V) \times_\nu V $ be the enhanced general linear group as in Corollary \ref{2.6}(1). Then $\lieg=\mathfrak{gl}(n)\ds V$.

\subsection{} For each $ x\in \mathfrak{gl}(n)$, we can write
	$ \det(t+x)=t^n+\sum_{i=0}^{n-1}\psi_i(x)t^i$. Then by the classical Chevalley restriction theorem   (see \cite[Proposition 7.9]{Jan})
	$$  \ff[\mathfrak{gl}(n)]^{\GL(n)}=\ff[\psi_0,\cdots,\psi_{n-1}].  $$
		For each $ 0\leq i\leq n-1, $  let $ \varphi_i:=\on{pr}^*(\psi_i) $. Namely,  $ \varphi_i(g):=\psi_i(x) $ for all $ g\in \lieg $ with $g=x+v$, $x\in \gl(n)$, $v\in V$. By Theorem \ref{che res thm}, 
	\[ S(\lieg^*)^{\lieg}=S(\lieg^*)^G=\ff[\lieg]^{G}= \ff[\varphi_0,\cdots,\varphi_{n-1}]. \]
	
	\subsection{}
	 Let $\mathfrak{p}$ be a parabolic subalgebra of maximal dimension in $\mathfrak{sl}(n+1)$ consisting of $(n+1)\times(n+1)$ zero-traced matrices with $(n+1,k)$-entries equal to zero  for any $ k=1,\ldots,n$.
	
	Consider two elements $ I=(\sum_{i=1}^{n} E_{ii},0)\in\ggg$ and $ H=\frac{1}{n+1}\sum_{i=1}^{n} (E_{ii}-E_{n+1,n+1})\in\liep$. Then  $ \lieg=\ff I\ds \pr{\lieg} $ and $ \liep=\ff H\ds \pr{\liep}$,  and $ \pr{\lieg}\iso \pr{\liep}\iso \mathfrak{sl}(n)\ds M$.   Moreover, one can check that there exists an isomorphism of Lie algebras: $ \lieg\iso \liep $,  sending $I$ to $H$.
	
	By \cite[Section 7]{Jo}, $ S(\liep)^{\pr{\liep}}=\ff[z] $ is a polynomial ring with one indeterminate where $ z $ is a harmonic element in $  \mathfrak{sl}(n+1). $ As in the proof of \cite[Lemma 7.4]{Jo},
$[H,z]=nz$. Therefore, we have
\begin{align}\label{key eq}
 S(\lieg)^{{\lieg}}\iso S(\liep)^{{\liep}}=\ff[z]^H=\ff.
\end{align}

We finally obtain the structure of the  center of $U(\ggg)$ for $\ggg=\gl(V)\ltimes V$.
\begin{prop} Let $\lieg=\gl(n)\oplus V$. Then the center of $U(\lieg)$ is  one-dimensional.
\end{prop}

\begin{proof} It follows from Duflo's theorem \cite{Duf} and \eqref{key eq} that
 $Z(\lieg)=U(\lieg)^{\lieg}\iso S(\lieg)^{\lieg}\iso\ff $. 
\end{proof}

\section{Application II: Nilpotent cone and Steinberg map for semi-reductive Lie algebras}
 Denote by $ \she{N} $ (resp. $ \she{N}_0 $) the variety  of all nilpotent elements of $ \lieg $ (resp. $ \lieg_0 $). Let $ \ff[\lieg]_+^G $ (resp. $ \ff[\lieg_0]_+^{G_0} $) be the subalgebra in $ \ff[\lieg]^G $ (resp. $ \ff[\lieg_0]^{G_0} $) consisting of all polynomials in $ \ff[\lieg]^G $ (resp. $ \ff[\lieg_0]_+^{G_0} $) with zero constant term. It is well-known that $ \nc_0$ coincides the common zeros of all polynomials $ f\in \ff[\lieg_0]_+^{G_0} $ (see \cite[Lemma 6.1]{Jan}).

 \subsection{}
 Suppose $G$ satisfies the positivity condition \ref{positive} and the assumption on ch$\bbf$  in Corollary \ref{2.5}. By this corollary, there exist  $n$ algebraically independent homogeneous polynomials $F_i\in \ff[\lieg]$, $i=1,\ldots,n$ such that $\ff[\lieg]^G =\ff[F_1,\cdots, F_n]$, of  which the precise meaning can be seen in Theorem \ref{che res thm}. Furthermore, we have the following lemma.

 \begin{lemma}\label{nil cone 2} Let $G$ be a semi-reductive algebraic group over $\bbf$ satisfying the positivity  condition \ref{positive} and the assumption on ch$\bbf$  in Corollary \ref{2.5}.
 Then
 	\[ \nc=\left\{ X\in \lieg\mid f(X)=0 \text{ for all } f\in \ff[\lieg]_+^G\right\} = V(F_1,\ldots,F_n). \]

 Here and further, the notation $V(F_1,\ldots,F_n)$ stands for the algebraic set of common zeros of $F_1,\ldots,F_n$.
 \end{lemma}
 \begin{proof}
 	Similar to \cite[Section 6.1]{Jan}, $ \nc\supseteq \left\{ X\in \lieg\mid f(X)=0 \text{ for all } f\in \ff[\lieg]_+^G\right\}. $ Actually Jantzen
 	only deal with the nilpotent cone for reductive Lie algebras in \cite{Jan}, but his	arguments remain valid for a general algebraic Lie algebra.
 	
 	If $ x\in \lieg $ is nilpotent, $ 0\in \overline{\co_X} $ by Proposition \ref{semisimple orbit}. For each $ F\in \ff[\lieg]_+^G, $ since $ F $ is continuous  and constant in $ \co_X, $ $ F $ is constant in $ \overline{\co_X}. $ Since $ F(0)=0, $ $ F(x)=0. $
 	
 	Thanks to \cite[Section 6.1]{Jan} and the definition of $ F_i, $
 	\[ \nc=V(F_1,\cdots,F_n)=V(f_1,\cdots,f_n)\times \uuu=\nc_0\times \uuu.\qedhere \]
 \end{proof}

So we have
\begin{prop}\label{nil cone 3} Keep the notations and assumption as in Lemma \ref{nil cone 2}. In particular,  $ G=G_0 \ltimes U$ denotes  a connected semi-reductive algebraic group over $ \ff $ satisfying the positivity condition \ref{positive}, and  $\lieg=\on{Lie}(G)$. The following statements hold.
	\begin{enumerate}
		\item $ \nc\iso \nc_0\times \uuu \text{ as varieties.} $
	\item For any $D\in \ggg$ with $D=x+y$, $\ x\in\lieg_0,\ y\in\uuu$, we have the following decomposition of tangent spaces
		\[ T_D(\nc) = T_x(\nc_0) \ds \uuu \]
		where $ \uuu $ is regarded as its tangent space. 
		\item Let $ \nc_{\text{sm}} $ $ \on{(} resp.\  \nc_{0,\text{sm}} \on{)}$  be the smooth locus of $ \nc $ $ \on{(} resp.\  \nc_0 \on{)}$. We have following isomorphism of varieties:
		\begin{align}\label{Iso Sm} \nc_{\text{sm}} \iso \nc_{0,\text{sm}}\times \uuu.
		\end{align}
		\item $ \nc $ is irreducible and normal.
	\end{enumerate}
\end{prop}
\begin{proof}
(1) folllows from the proof of Lemma \ref{nil cone 2}.

(2)  It is a direct consequence of  (1).
	
(3) Thanks to (1) and (2), $\dim(\nc) = \text{dim}(\nc_0) + \text{dim}\,\uuu $ and
	$$ \dim T_X(\nc) =\dim T_{\on{pr}(X)}(\nc_0) + \text{dim}\,\uuu,\,\,\forall\,  X\in \nc.$$
	Hence, $ X\in \nc $ is smooth if and only if $ \on{pr}_0(X)\in \nc_{\text{0}}$ is smooth.
	As a result, the restriction of $ \alpha $  to  $ \nc_{\text{sm}} $ gives rise to  the isomorphism   (\ref{Iso Sm}).
	
(4) Since $ \nc_0 $ is irreducible and normal, so is $\nc$ by (1).
\end{proof}

\subsection{The Steinberg map}
Define canonically an adjoint quotient map
\begin{align*}
\chi:  \lieg &\rightarrow \bb{A}^n, \cr
 D&\mapsto (F_1(D),\cdots, F_n(D)),\quad\forall\,D\in\lieg
\end{align*}
which we call the Steinberg map for $\ggg$. Let
\begin{align*}
\eta:  \lieg_{0} &\rightarrow \bb{A}^n, \cr
 x&\mapsto (f_1(x), \cdots,f_n(x)),\quad\forall\,x\in\lieg_0
\end{align*}
be the Steinberg map for $ \lieg_{0} $.  Since $ F_i=\on{pr}^*(f_i) $, we have $ \chi=\eta\circ \on{pr}. $

\begin{prop}\label{5.3}
	 Keeping the notations and assumptions as in Corollary \ref{nil cone 2}. In particular,  $G=G_0\ltimes U$ be a connected semi-reductive  algebraic group with $\lieg=\Lie(G)=\lieg_0\oplus \uuu$ satisfying the positivity condition.
	\begin{enumerate}
		\item $ \chi^{-1}(0)=\nc_0+\uuu. $ In particular, $ \chi^{-1}(0)=\nc $. 
		\item $ \chi^{-1}(a)=\eta^{-1}(a)\times \uuu$ for all $ a\in \chi(\lieg). $ In particular, $ \chi^{-1}(a) $ is irreducible of dimension $ \on{dim}(\lieg)-n. $
		\item For every $ a\in \chi(\lieg), $ $ \chi^{-1}(a) $ contains exactly one orbit consisting of semisimple elements.
	\end{enumerate}
\end{prop}
\begin{proof}
(1) and (2) are obvious.

(3) Fix a maximal torus $ \liet $ of $ \lieg_0$. Then $ \liet $ is a maximal torus of $ \lieg$.  Thanks to Lemma \ref{startpoint}, for all semisimple elements $ x_1, x_2\in \chi^{-1}(a), $ there exist $ g_i\in G,\ i=1,2, $ such that $ y_i:=g_i\cdot x_i\in \liet\subseteq \lieg_0,\ i=1,2. $ Note that $ \eta(y_i) =\chi(y_i) = a,\ i=1,2, $ and $ \lieg_0 $ is reductive. It follows from \cite[Proposition 7.13]{Jan} that $ y_1 $ and $ y_2 $ lie in the same $ G_0 $-orbit, namely, there is $ g_0\in G_0 $ such that $ g_0\cdot y_1=y_2. $ Moreover, $ x_2=g_2^{-1}\cdot g_0\cdot g_1\cdot x_1. $ As a result, $ x_1 $ and $ x_2 $ are in the same $ G $-orbit. Hence (3) holds.
\end{proof}

\subsection{Springer resolution for the semi-reductive case}
Let $ \she{B} $ (resp. $ \she{B}_0 $) be the set of all Borel subalgebras of $ \lieg $ (resp. $ \lieg_0 $). Lemma \ref{startpoint} implies that $ \lieb_0\mt \lieb_0\ds \uuu $ defines an isomorphisms between $ \she{B}_0 $ and $ \she{B}. $

Recall that for fixed positive root system $ \Delta^+, $ denote by $ B^+=B_0^+\ltimes U $ (resp. $ \lieb^+=\lieb_0^+\ds \uuu $) the corresponding Borel subgroup of $ G $ (resp. Borel subalgebra of $ \lieg $) where $ B_0^+ $ (resp. $ \lieb_0^+ $) is the corresponding Borel subgroup of $ G_0 $ (resp. Borel subalgebra of $ \lieg_0 $) (see \S\ref{1.2}). Set	$ \tilde{\she{N}}:=\{ (x,\mathfrak{b})\in \she{N}\times \she{B}\mid x\in \mathfrak{b} \}$.
Similar to the case of reductive Lie algebras, we have the following lemma.
\begin{lemma}
	Keep notations as above, we have
	\begin{enumerate}
		\item $ N_G(\lieb^+)=B^+. $
		\item $ \she{B}\iso G/B^+, $ and $ \she{B} $ is a projective variety.
		\item $ \tilde{\she{N}}\iso \{ (x,gB^+)\in \she{N}\times G/B^+\mid g^{-1} x\in \lieb^+ \}. $
	\end{enumerate}
\end{lemma}
\begin{proof}
(1)  For each $ x\in \lieb_0,\ y\in \uuu $, suppose $ g\in G_0, h\in U $ such that $ gh(x+y)\in \lieb. $ Note that $ h(x)=x+\pr{x}$, where $ \pr{x}\in \uuu. $ Therefore, $ gh(x+y)=g(x)+g(\pr{x}) +gh(y)$. Hence, $ g(x)\in \lieb_0 $ and $ g(\pr{x}) +gh(y)\in\uuu. $ It follows that $ g\in N_{G_0}(\lieb_0)=B_0^+ $ and $ N_G(\lieb^+)\leq B_0^+\ltimes U=B^+. $
Since $ \on{Lie}(B^+)=\lieb^+,\ B^+\leq N_G(\lieb^+). $ Consequently, $ N_G(\lieb^+)=B^+. $

(2) follows from (1) and Lemma \ref{startpoint}.

(3) follows from (2).
\end{proof}

Denote by $ \nnn $ the nilpotent radical of $ \lieb^+ $, then $ \nnn=\nnn_0 \ds \uuu$ where $ \nnn_0 $ is the nilpotent radical of $ \lieb^+_0. $ One can check that $ B^+ $ can act on $ G\times \nnn$, with $ x\cdot (g,n)=(x\cdot g^{-1},x\cdot n)$ for $x\in B^+$ and $g\in G$, $n\in \nnn$.   
The arguments of \cite[\S 6.5]{Jan} still work in our case. Therefore, the following proposition holds.

\begin{proposition}\label{smoothness of N-tilde}
	The projection $ \pi: \tilde{\she{N}}\ra \she{B} $ makes $ \tilde{\she{N}} $ a $ G $-equivariant vector bundle over $ \she{B} $ with fiber $ \nnn. $ The assignment $ (g,n)\mt (g\cdot n, gB^+) $ gives a $ G $-equivariant isomorphism
	$ G\times^{B^+}\nnn\iso \tilde{\she{N}}. $ In particular, 	$ \tilde{\she{N}} $ is smooth.
\end{proposition}

\begin{lemma}\label{proper of mu}
	The projection $ \mu:\tilde{\she{N}}\ra \she{N} $ is a proper map.
\end{lemma}
\begin{proof}
	By definition, $ \mu=\beta\circ \alpha, $ where $ \alpha:\tilde{\she{N}}\inj \she{N}\times \she{B} $ is a closed immersion and $ \beta: \she{N}\times \she{B} \ra \she{N} $ is projection. Since $ \she{B} $ is a projective variety, $ \mu $ is a projective morphism. Therefore $ \mu $ is proper.
\end{proof}

Let $ \she{B}_X:=\{ \lieb\in\she{B}\mid X\in\lieb \} $ for $ X\in\nc $ and $ \she{B}_{0,x}:=\{ \lieb_0\in\she{B}_0\mid x\in\lieb_0 \} $ for  $ x\in \nc_0. $

\begin{lemma}
	Suppose $ X=X_0+X_1\in\nilg{} $ where $ X_0\in\lieg_0, X_1\in\uuu. $ Then $ \psi: \lieb_0\mt \lieb_0 \ds \uuu $ defines a one-to-one correspondence between $ \she{B}_{0,X_0}$ and $\she{B}_X. $
\end{lemma}
\begin{proof}
	Note that  $ X\in\lieb_0\ds \uuu $ if and only if $ X_0\in \lieb_0. $ The assertion follows immediately.
\end{proof}

\begin{lemma}\label{birational equivalence}
	If $ n\in\she{N}_{\text{sm}}, $ then $ \mu^{-1}(n) $ contains exactly one element. Namely, there is unique Borel subalgebra of $ \lieg $ contains $ n. $
\end{lemma}
\begin{proof}
It follows from Proposition  \ref{nil cone 3} that $ n=n_0+n_1 $ for $ n_0\in \nilg{0}_{\text{sm}} $ and $ n_1\in \uuu. $ Since $ \lieg_0 $ is a reductive Lie algebra, there is exactly one Borel subalgerba $ \lieb_0\in \lieg_0 $ such that $ n_0\in \lieb_0. $
	
	Now suppose $ n\in \lieb $ for a Borel subalgerbra $\lieb$ of $ \lieg, $ and $ \lieb=\pr{\lieb}\ds\uuu $ where $ \pr{\lieb} $ is a Borel subalgebra of $ \lieg_0. $ Since $ n_0\in \pr{\lieb}, $ then $ \pr{\lieb}=\lieb_0 $ and $ \lieb=\lieb_0\ds \uuu. $
\end{proof}

Thanks to Proposition \ref{smoothness of N-tilde}, Lemmas \ref{proper of mu} and \ref{birational equivalence}, $ \tilde{\she{N}} $ is smooth and $ \mu $ is a proper birational morphism. As a consequence, the following result holds.

\begin{prop}\label{5.9}
	Keep notations as above, $ \mu:\tilde{\she{N}}\ra \she{N} $ is a resolution of singularities.
\end{prop}

\begin{remark} Let $ \tilde{\nc_0}=\{ (x,\mathfrak{b})\in \she{N}_0\times \she{B}_0\mid x\in \mathfrak{b} \} $ be the Springer resolution of $ \nc_0. $
	Assume $ G $ satisfies the positivity condition \ref{positive} and the assumption of ch$\bbf$ as in Corollary \ref{nil cone 2},  then
	$$ \tilde{\nc}\iso \tilde{\nc_0}\times \uuu.$$
Actually, by Proposition \ref{nil cone 3}, $ \nc=\nc_0\times \uuu. $ Then $ (x,\lieb_0,y)\mt (x+y,\lieb_0+\uuu) $ defines an isomorphism between $ \tilde{\nc_0}\times \uuu $ and $ \tilde{\nc}$.
\end{remark}

\end{document}